\documentclass[a4paper,11pt]{amsart}
\usepackage{amsmath}
\usepackage{graphicx}
\usepackage{epstopdf}
\usepackage{amssymb,amscd,array}
\usepackage{color}
\usepackage{float}

\newcounter{tablegroup}
\newcounter{subtable}[tablegroup]

\newtheorem{thm}{Theorem}[section]
\newtheorem{cor}[thm]{Corollary}
\newtheorem{lem}[thm]{Lemma}

\numberwithin{equation}{section}

\begin{document}
\title[Maximal equicontinuous factor and minimal map on finitely suslinean continua]
{Maximal equicontinuous factor and minimal map on finitely suslinean continua}

\author{$\textmd{Aymen\: Daghar}^{1,2}$}

\address{1.University of Carthage, faculty
of sciences of bizerte, (UR17ES21), "Dynamical Systems and their Applications"
\\ 7021, Jarzouna, Tunisia}

\address{2.University of Carthage, higher institut of management Bizerte, Menzel Abderrahmen (7035) , Bizerte, Tunisia;}

\email{aymen.daghar@isgb.ucar.tn and aymendaghar@gmail.com}

\begin{abstract}
In this paper, we introduce the notion of negatively regionally proximal pairs of onto maps which coincides with the set of regionally proximal pair of $f^{-1}$, whenever $f$ is an homeomorphism and we prove the maximal equicontinoues factor for any onto map on a locally connected continuum is monotone. Using this, we prove that if $f$ is a minimal map on a finitely suslinean continua $X$, then $X$ must be a topological circle and $f$ some irrational rotation of circle.
\end{abstract}
\maketitle

\section{Introduction}

A \textit{continuum} is a non-empty compact metric connected space. A continuum is \textit{finitely suslinean} if any family of pairwise disjoint subcontinuum is null. Finitely suslinean continua are known to be one-dimensional Peano-continua. Note that the class of finitely suslinean is larger than those regular curves, local dendrite, graph and dendrite (see \cite{Kura}).\\
Many authors where interested in the study of minimal maps, see for instance (\cite{M1,M2,M3}). Among the question related to this subject, we can wonder whenever a given space admit a minimal homeomorphism (\cite{M5,M6,M4}) and whenever a space admits a minimal map that's not an homeomorphism, such space exists as shown in many papers, such that \cite{M3}. Among finitely suslinean continua, the circle $S_1$ is the only one known to have a minimal map, moreover this map must be an homeomorphism. It has been proven that any topological graph that admit a minimal map is the circle $S_1$ \cite{G1}, this result is also true whenever $X$ is a local dendrite \cite{LC1}. Since those continua are in particular finitely suslinean, it is only naturel to question if this result stay true in the case where $X$ is finitely suslinean continua.\\
In this paper, we deal with this question and we prove that this result hold for finitely suslinean continua. For the proof, we show that the maximal equicontinoues factor for any onto map on a locally connected continuum is monotone, this was already proven in the case of homeomorphism and more generally in the case of group action, (see \cite{MEQ}, recall that a map is said to be monotone if the preimage of any connected subset if also connected). For that end, we introduce first the notion of negatively regionally proximal pairs of onto maps which coincides with the set of regionally pair of $f^{-1}$, whenever $f$ is an homeomorphism. Then, we prove that if $f$ is a minimal map on a finitely suslinean continua, $f$ cannot be weakly mixing and recalling that in the case of minimal maps, either $f$ has a non-trivial maximal equicontinous factor or $f$ is weakly mixing.\\

\section{Preliminaries}

Let $X$ be a compact metric space with metric $d$. The closure (resp. the boundary) of a subset
$A$ of $ X$ is denoted by $\overline{A}$ (respectively $ \partial A$), for $C\subset A$, we denote by $Bd_{A}(C), Cl_{A}(C), int_{A}(C)$ respectively the boundary, the closure and the interior of $C$ in $A$ endowed with the induced topology from the one on $X$.  A subset of $X$ is said to be clopen if it is open and closed at the same time. We denote by $2^X$ (resp. $C(X)$) the set of all non-empty compact subsets (resp. compact connected) subsets of $X$ and we endow $2^{X}$(resp. $C(X)$) with the Hausdorff metric $d_H$ defined as follows: $$d_H (A,B)=\max (\sup_{a\in A} d(a,B),\sup_{b\in B} d(b,A) ) $$ where $A, B \in 2^X$ and $d(x,M)=\inf_{y \in M} d(x,y)$ for any $x\in X$ and $M\in 2^X$. We denote also by $d_{I}(A,B)=\inf\{d(x,y),x\in A,y\in B\}$. Notice that both $(C(X), d_H)$ and  $(2^X, d_H)$ are compact metric space (for more details see \cite{conti}).\\
A sequence $(A_{n})_{n\geq 0}$ of compact subset of $X$ converges to $A\in 2^{X}$, if $\displaystyle\lim_{n\to \infty}d_{H}(A_{n},A)=0$. 
Let $A\subset X,\; \delta(A):=\sup_{x,y\in A}d(x,y)$.\\
Let $f: X\to X$ be a continuous map of $X$. We denote by $2^{f}:2^{X}\to 2^{X}, A\to f(A)$ and by $C(f)=2^{f}_{\setminus C(X)}$. Note that $2^{f},C(f)$ are continuous self mapping of $(2^{X},d_{H})$ and$(C(f),d_H))$ (see \cite{conti}). During the paper $S_1$ is any topological circle (i.e a topological space homeomorphic to the set $\{(x,y)\in \mathbb{R}^2,\; x^2+y^2=1\}$.\\

\begin{itemize}
  \item A family $\{A_{i},\; i\in I\}$ of subset of a metric space $X$ is said to be \emph{a null family} if for any $\epsilon>0$ there exist a finite subset $J$ of $I$ such that for any $i\in I\setminus J$, we have $\delta(A_{i}) < \epsilon $.
  \item A \textit{regular curve} (resp. a \textit{rational curve}) is a continuum for which any point has an $\epsilon$-open neighbourhood with finite (resp. at most countable) boundary, for any $\epsilon>0$.
  \item A \emph{finitely suslinean continuum} is a continuum for which any pairwise disjoint family of subcontinua is null.
  \item  A \emph{hereditarily locally connected continuum} is a locally connected continuum such that any sub-continuum of which is locally connected.
  \item A \emph{Suslinean continuum} is a continuum for which any pairwise disjoint family of subcontinua is at most countable.

\end{itemize}

Recall that each regular curve is finitely suslinean and a finitely suslinean continuum is a hereditarily locally connected continuum. Note that a hereditarily locally connected continuum is a locally arcwise connected rational continuum and any rational continuum is in particular suslinean (for more details \cite{conti}, \cite{Kura}).

A dynamical system is a pair $(X, f)$, where $X$ is a compact metric space and $f : X\to X$ is a continuous map. Let $\mathbb{Z},\ \mathbb{Z}_{+},\; \mathbb{Z}_{-}$ and $\mathbb{N}$ be the sets of integers, non-negative integers, non-positive integers and positive integers, respectively.\\
Let $(X,f)$ be a dynamical system. The forward orbit of a given point $x\in X$ is the subset $O_{f^{+}}(x)=\{f^n(x): n\in\mathbb{Z}_{+}\}$.
The $\omega$-limit set of a given point $x\in X$ is defined as follow:
$$\omega_f(x)=\cap_{n\in\mathbb{N}}\overline{\{f^k(x): k\geq n\}}$$

$$=\{y\in X: \exists n_1<n_2<\cdots: \lim_{i \rightarrow +\infty} d(f^{n_i}(x),y)=0\}$$

A point $x$ in $X$ is called:
\begin{itemize}
\item \emph{fixed} if $f(x)=x$,
\item \emph{periodic} if $f^n (x)=x$ for some $n\in \mathbb{N}$,
\end{itemize}

We denote by $P(f)$ the sets of periodic points. A subset $M$ is called \emph{minimal} if it is not empty, closed, $f$-invariant (\textit{i.e.} $f(A)\subset A$) and there is no proper subset of $M$ having these properties. A subset $M$ of $X$ is called \emph{totally minimal} if it is $f^n$-minimal for all $n\in \mathbb{N}$.
The system $(X,f)$ is said to be transitive (resp. weakly mixing) if for any two non-empty open sets $U,V$ of $X$ there exists $n\in \mathbb{N}$ such that  $f^{n}(U)\cap V\neq\emptyset$ (resp. $f^{2}$ is transitive).\\
Let $A$ be a subset of $X$, we denote by $\Delta_{A}=\{(x,x),\; x\in A\}$. A pair $(x,y)\in X\times X$ is called \textit{proximal} if $\underset{n\to \infty}\liminf\ d(f^n(x),f^n(y))=0$ otherwise it is called \textit{distal}. If $\underset{n\to \infty}\limsup\ d(f^n(x),f^n(y))=0$, then $(x,y)$ is called \textit{asymptotic}. A pair $(x,y)$ is called a \textit{Li-Yorke pair} if it is proximal but not asymptotic. The dynamical system $(X,f)$ is called \textit{distal} (resp. \textit{almost distal}) if it has no proximal pairs (resp. no Li-Yorke pairs). A point $a$ of $X$ is said to be distal if $(x,a)$ is a distal pair for any $x\in X$.
 A pair $(x,y)\in X\times X$ is called \textit{regionally proximal} if for each $\epsilon>0$ and each pair of open sets $O_{x},O_{y}$ containing respectively $x$ and $y$, there exist $x^{\prime}\in O_{x},\; y^{\prime}\in O_{y}$ and some $n\in \mathbb{N}$ such that $d(f^{n}(x^{\prime}),f^{n}(y^{\prime}))<\epsilon$.\\
 Denote by $P(X,f)$ the set of proximal pairs, $A(X,f)$ the set of asymptotic pairs and by $RP^{+}(X,f)$ the set of regionally proximal pairs. When $f$ is an homeomorphism we denote by $RP_{\mathbb{Z}}(X,f)=RP^{+}(X,f)\cup RP^{+}(X,f^{-1})$. Clearly $A(X,f),\; P(X,f)$ and $RP^{+}(X,f)$ are $f\times f$ invariant.
 \bigskip

 Given two dynamical systems $(X,f)$ and $(Y,g)$, by \textit{factor map} we mean a continuous surjective $\pi: X\to Y$ satisfying $\pi \circ f = g \circ \pi$, in this case $(X,f)$ is called an extension of $(Y,g)$ and $(Y,g)$ is called a factor of $(X,f)$. Moreover if $\pi$ is an homeomorphism we say that $(X,f)$ and $(Y,g)$ are topologically conjugated.

Recall that a $\mathbb{Z}$-action on $X$ generated by a self homeomorphism $f$ of $X$ is said to be \textit{equicontinoues} if the family $\{f^{n},\; n\in \mathbb{Z}\}$ is equicontinoues.\\

Let $(X,f)$ be a dynamical system, where $f$ is a self homeomorphism of $X$, the maximal equicontinous factor $(Y,g)$ of $(X,f)$ is an equicontinous  factor of $(X,f)$ such that any equicontinous factor $(Z,h)$ of $(X,f)$ is a factor of $(Y,g)$.\\
For $x\in X$, let
 $RP^{-}(x)=\{y\in X,\; \exists (x_n,y_n) \in X^2,\; m_n \to  +\infty\; such\; that:\\ \; d(x_n,y_n)\to 0 \;and\; f^{m_n}(x_n)\to x,\; f^{m_n}(y_n)\to y\}.$\\
 Denote by $S_{eq}(X,f)$ the equivalence relation on $X$ such that the factor induced in the maximal equicontinoues factor.  Note that for any dynamical system $(X,f),\; S_{eq}(X,f)$ exists and $P(X,f)\subset RP^{+}(X,f)\subset S_{eq}(X,f)$.  Therefore the maximal equicontinoues factor will always be an equicontinoues homeomorphism, (see \cite{eqq}) \cite{RP}).\\
For $x\in X$, we denote by $S_{eq}(x)=\{y\in X,\; (x,y)\in S_{eq}(X,f)\}$ and $RP^{+}(x)=\{y\in X,\; (x,y)\in RP^{+}(X,f)\}$.

\section{Monotonie of the maximal equicontinoues factor}
In this section, we prove that the maximal equicontinous factor for any onto map is monotone. Recall that the maximal equicontinoues factor is always an equicontinoues homeomorphism.

 \begin{lem}\label{RP- inclus}
  Let $(X,d)$ be a compact metric space and $f:X\to X$ be an onto continuous map. Let $x \in X$, then $RP^{-}(x)\subset S_{eq}(x)$.
\end{lem}

 \begin{proof}
 We will assume that $S_{eq}\subsetneq X^{2}$, otherwise the inclusion is trivial. Let $x\in X$ such that $RP^{-}(x) \nsubseteq S_{eq}(x)$, then there exists $y\in X$ such that $(x,y)\in RP^{-}(X,f)$ and $(x,y)\notin S_{eq}(X,f)$. $S_{eq}(X,f)$ is a closed equivalence relation and the factor system $\tilde{f}$ is an equicontinous homeomorphism, thus the family $\{\tilde{f}^{n},\; n\in \mathbb{Z}\}$ is equicontinous. We have $(x,y)\in RP^{-}(X,f)$, then we can find a sequence $(x_n,y_n)_{n\geq 0}$ converging to $(x,y)$ and $(m_n)_{n\geq 0}$ an increasing sequence of integer $a_n\in f^{-m_n}(x_n)$ and $b_n\in f^{-m_n}(y_n)$ such that $d(a_n,b_n))_{n\geq 0}$ converges to $0$. Recall that $(x,y)\notin S_{eq}(X,f)$, thus $\tilde{x}\neq \tilde{y}$ and for $n$ large enough $\tilde{x_n}\neq \tilde{y_{n}}$. Since $f^{m_n}(a_n)=x_n$ and $f^{m_n}(b_n)=y_n$ then $\tilde{f}^{m_n}(\tilde{a_n})=\tilde{x_n}$ and $\tilde{f}^{m_n}(\tilde{b_n})=\tilde{y_n}$. We have $d(a_n,b_n))_{n\geq 0}$ converges to $0$, then so does $\tilde{d}(\tilde{a_n},\tilde{b_n})_{n\geq 0}$ then $(\tilde{x},\tilde{y})\in RP(\tilde{X},\tilde{f}^{-1})$. Recall that $(\tilde{X},\tilde{f})$ is an equicontinoues homeomorphism then $RP(\tilde{X},\tilde{f}^{-1})=RP(\tilde{X},\tilde{f})=\Delta_{\tilde{X}}$ and therefore $\tilde{x}=\tilde{y}$, a contradiction.
 \end{proof}

\begin{lem}\label{RP- conn}
  Let $X$ be a locally connected continuum and $x \in X$, then $RP^{-}(x)$ is connected.
\end{lem}
`\begin{proof}

Let $x\in X$ and $C_x$ the connected component of $RP^{-}(x)$ and $y\in RP^{-}(x)$ such that $y\notin C_x$.

\textbf{Claim:} There exists an open set $U$ of $X$ such that $RP^{-}(x) \subset U$, in which the connected components of $U$ containing respectively $x$ and $y$ are disjoint.

Indeed, since $RP^{-}(x)$ is a compact subset we can find a sequence of open neighbourhoods $(O_n)_{n \geq 0}$ of $RP^{-}(x)$ such that
\[
RP^{-}(x) = \bigcap_{n \geq 0} O_n.
\]
If for infinitely many $n \geq 0$, $x$ and $y$ are in the same connected component of $O_n$, namely $C_n$, then there is a subcontinuum $I \subset RP^{-}(x)$ containing $x$ and $y$ (it suffices to select a suitable converging subsequence of $(C_n)_{n \geq 0}$), a contradiction.

Hence, for some $N$, the connected components of $O_n$ containing respectively $x$ and $y$ are disjoint for any $n \geq N$. Consider then $U = O_N$. This ends the proof of the claim.

As $y\in RP^{-}(x)$, let $(x_n)_{n \geq 0}$ (resp. $(y_n)_{n \geq 0}$) be a sequence of $X$ converging to $x$ (resp. $y$) and $(m_n)_{n \geq 0}$ a sequence of integers such that

$\lim_{n \to +\infty} d_{I}(f^{-m_n}(x_n), f^{-m_n}(y_n)) = 0.$
Denote by $C_{x, U}$ and $C_{y, U}$ the connected component of $U$ containing respectively $x$ and $y$. Since $U$ is an open set of $X$ which is locally connected, and by the claim above, $C_{x, U}$ (resp. $C_{y, U}$) are disjoint open neighbourhoods of $x$ (resp. $y$). Therefore, for $n$ large enough, $x_n \in C_{x, U}$ and $y_n \in C_{y, U}$.

Since
\[
\lim_{n \to +\infty} d_{I}(f^{-m_n}(x_n), f^{-m_n}(y_n)) = 0
\]
and $X$ is a locally connected continuum, we may find a sequence of arcs $(J_n)_{n \geq 0}$ joining a point $a_n\in f^{-m_n}(x_n)$ and $b_n \in f^{-m_n}(y_n)$ such that
\[
\lim_{n \to +\infty} \text{diam}(J_n) = 0.
\]
Observe that for $n$ large enough, $f^{m_n}(J_n)$ is an continuum containing $x_n$ and $y_n$, thus meeting $C_{x, U}$ and $C_{y, U}$. Therefore, $f^{-m_n}(J_n) \nsubseteq U$, otherwise $C_{x, U} = C_{y, U}$. Hence, for some $N \geq 0$ and for any $n \geq N$, $f^{m_n}(J_n) \cap \partial U = \emptyset$.

Let then $t_n \in f^{m_n}(J_n) \cap \partial U$, thus $f^{-m_n}(t_n) \cap J_n \neq \emptyset$. It turns out that
\[
\lim_{n \to +\infty} d_{I}(f^{-m_n}(x_n), f^{-m_n}(t_n)) = 0.
\]
Let $t \in X$ be some limit point of the sequence $(t_n)_{n \geq N}$, clearly $t \in \partial U$, moreover $t \in RP^{-}(x)$, therefore $t \in RP^{-}(x) \cap \partial U$, which is a contradiction, since $U$ is an open neighbourhood of $RP^{-}(x)$.
\end{proof}

\begin{thm}\label{factor monotone}
    Let $(X,f)$ be a continuous map on a locally connected continuum, then for any $x\in X$, $S_{eq}(x)$ is connected. In particular $(X,f)$ is semi-conjugated to it's maximal equicontinoues factor via a monotone map.
\end{thm}

\begin{proof}
Since $S_{eq}(X,f)$ is an equivalence relation, we can assume that for any $x\in X$ $S_{eq}(x) \subsetneq X$, otherwise we are done since $X$ is connected.\\
  Consider the equivalence relation $\mathcal{R}$ collapsing the connected component of each $S_{eq}(x)$. Clearly $\mathcal{R}$ is closed and the quotient map $\pi$ is monotone. Let $(\tilde{X},\tilde{f})$ be the factor system.\\
  
 \textbf{Claim} $(\tilde{X},\tilde{f})$ is equicontinous.\\

 Assume this is not the case, then we can find a sequence $(\tilde{x_n})_{n\geq 0}$ converging to $\tilde{x}$ and $(m_n)_{n \geq 0}$ a sequence of integers such that $\lim_{n \to +\infty} \tilde{d}(\tilde{f^{m_n}}(\tilde{x_n}),\tilde{z_1})=0$ and $\lim_{n \to +\infty} \tilde{d}(\tilde{f^{m_n}}(\tilde{x}),\tilde{z_2}) = 0$, with $\tilde{z_1}\neq \tilde{z_2}$. Let n$\geq 0$ and $(x,z_1,z_2,x_n) \in X^4$ such that $\pi(x,z_1,z_2,x_n)=(\tilde{x},\tilde{z_1},\tilde{z_2},\tilde{x_n})$. We can assume up to taking a sub-sequence that $(\pi(x_n))_{n\geq 0}$ converges to $A\subset \pi(x)$, $\pi(f^{m_n}(x_n))_{n\geq 0}$ converges to $Z_1\subset \pi(z_1)$ and $\pi(f^{m_n}(x))_{n\geq 0}$ converges to $Z_2\subset \pi(z_2)$, with respect to the Hausdorff metric.\\
 Let now $c_n\in \pi(x_n)$, the sequence $(c_n)_{n\geq 0}$ converge to some $c\in \pi(x)$. On one hand $f^{m_n}(c)\in \pi(f^{m_n}(x))$ on the other hand $f^{m_n}(c_n)\in \pi(f^{m_n}(x_n))$. Therefore $(f^{m_n}(c_n))_{n\geq 0}$ converges to $t_1 \in \pi(Z_1)$ and $(f^{m_n}(c))_{n\geq 0}$ converges to $t_2 \in Z_2$. Therefore $(t_1,t_2)\in RP^{-}(X,f)$. By Lemma \ref{RP- conn}, we have that $RP^{-}(t_1)$ is connected and by Lemma \ref{RP- inclus}, we have that $RP^{-}(t_1)\subset S_{eq}(t_1)$. We conclude that $t_1,t_2$ are in the same connected component of $S_{eq}(t_1)$ thus $(t_1,t_2)\in \mathcal{R}$, thus $\pi(t_1)=\pi(t_2)$. Recall that $t_1\in Z_1\subset \pi(z_1)$ and  $t_2\in Z_2\subset \pi(z_2)$, thus $\pi(z_1)=\pi(z_2)$ a contradiction and the claim is proven, thus $S_{eq}(X,f) \subset \mathcal{R}$.\\

 It remains to show that for any $S_{eq}(x)$ is connected, this follows immediately from the fact that: $$\mathcal{R}\subset S_{eq}(X,f) \subset \mathcal{R}.$$ Therefore for any $x\in X,\; S_{eq}(x)=\mathcal{R}(x)$ is connected and the map $\pi$ is monotone.
 \end{proof}
 \begin{cor}
   Let $(X,f)$ be a minimal map on a locally connected continuum, then either $(X,f)$ is weakly mixing or $(X,f)$ is semi-conjugated via a monotone map to some isometry $(\tilde{X},\tilde{f})$, where $\tilde{X}$ is an homogenous locally connected continuum.
 \end{cor}

It is well known that the simple closed curve is the only homogenous suslinean continua (see \cite{hsc}), moreover suslinean continua are invariant under monotone mapping. As immediate consequence we have the following :
\begin{thm}
Let $(X,f)$ be a minimal map on a locally connected suslinean continua, then either $(X,f)$ is weakly mixing or $(X,f)$ is semi-conjugated to an irrational rotation of the circle
\end{thm}

  \section{Minimal map on finitely suslinean continua}
 In the following we will show that minimal map on finitely suslinean continua can't not be weakly mixing thus semi-conjugated via a monotone map to some isometry $(\tilde{X},\tilde{f})$, where $\tilde{X}$ is an homogenous finitely suslinean continuum, thus $\tilde{X}$ has to be a topological cirle $S_1$. We will later show that $X=S_1$.

\begin{thm}\cite{RP}\label{RPP}
 Let $f:X\to X$ be a minimal map, where $X$ is compact metric space. We have the following : \begin{itemize}
                                                                                               \item $RP^{+}(X,f)=S_{eq}(X,f).$ (\emph{ Theorem 1.4})
                                                                                               \item $RP^{+}(X,f)$ is a closed  invariant equivalence relation.
                                                                                               \item The Induced factor system in an equicontinoues homeomorphism (\emph{Theorem 1.4})
                                                                                               \item $RP^{+}(X,f)=X^2 \Leftrightarrow (X,f)$ weakly mixing. \emph{(Theorem 4.12)}
                                                                                               \end{itemize}
                                                                                               \end{thm}

 \begin{thm}\cite{PCWE}(Theorem 5.3)\label{PCW}
  If $f:X\to X$ is a positively continuum-wise expansive map of a compact metric space $X$, then every minimal set of $f$ is $0$ dimensional.
\end{thm}

\begin{thm}\cite{TP}
Let $f:X\to X$ be a minimal weakly mixing map, where $X$ is compact metric space
, then for every $x\in X$ the set $Tran(x)\subset X$ of points $y$ such that $(x,y)$ is a transitive point with respect to $f^2$, is a dense $G_{\delta}$ set.

\end{thm}

\begin{thm}\label{NWM}
  Let $f:X\to X$, be a minimal map where $X$ is finitely suslinean continua, then $f$ is not weakly mixing.
\end{thm}

\begin{proof}
  We assume that $f$ is weakly mixing and we will get a contradiction. \\

  First let $\mathcal{R}$ be the relation defined as: $$x\mathcal{R}y \Leftrightarrow \exists C\in C(X),\; \{x,y\}\subset C\; \emph{and}\: \displaystyle\lim_{n\to +\infty}\delta(f^{n}(C))=0.$$

  \textbf{Claim 1} $\mathcal{R}$ is a closed invariant monotone equivalence relation:\\
  It is clear that $\mathcal{R}$ is an invariant equivalence relation, we will first show that for any $x\in X,\; \pi(x)\in C(X)$. Observe that the family $(f^n(\pi(x)))_{n\geq 0}$ is pairwise disjoint, otherwise for some $n<m$ we have that $f^{n}(\pi(x))\cap f^{m}(\pi(x))\neq \emptyset$. Thus we can find $x_1,x_2 \in \pi(x)$ such that $f^{n}(x_1)=f^{m}(x_2)$. Thus for any $k\geq 0,\; f^{n}(f^{k}(x_1))=f^{m}(f^k(x_2))$. Let $(m_k)_{k\geq 0}$ an increasing sequence of integer such that $(f^{m_k}(x_1))_{k\geq 0}$ converges to some $c\in X$. Since $\pi(x)\times \pi(x) \subset Asm(X,f)$, we have that $(x_1,x_2)\in Asm(X,f)$, therefore $(f^{m_k}(x_2))_{k\geq 0}$ converges also to $c\in X$. For any $k\geq 0$, we have $f^{n}(f^{m_k}(x_1))=f^{m}(f^{m_k}(x_2))$, thus $f^{n}(c)=f^{m}(c)$ with $n\neq m$, therefore $P(f)\neq \emptyset$ a contradiction. We conclude that the family $(f^n(\pi(x)))_{n\geq 0}$ is pairwise disjoint. Observe that $\pi(x)$ is arc-wise connected, in fact for any $y\in \pi(x)$, there exist $C\subset \pi(x)$ with $\{x,y\} \subset C \subset \pi(x)$, with $C\in C(X)$ (Recall that any subcontinuum of $X$ is arc-wise connected). Therefore $(f^n(\pi(x)))_{n\geq 0}$ is a family of pairwise disjoint arc-wise connected subset of $X$. If $\displaystyle\lim_{n\to +\infty}\delta(f^{n}(\pi(x))\neq 0$, then we can find a sequence of pairwise disjoint arc $(I_n)_{n\geq 0}$ such that $I_n \subset f^{n}(\pi(x))$ and $\displaystyle\lim_{n\to +\infty}\delta(I_n)\neq 0$, a contradiction since $X$ is finitely suslinean. We conclude that $\displaystyle\lim_{n\to +\infty}\delta(f^{n}(\pi(x))=0$, therefore $\displaystyle\lim_{n\to +\infty}\delta(f^{n}(\overline{\pi(x)})=0$, moreover $\overline{\pi(x)}\in C(X)$ since $\pi(x)$ is connected. We conclude that $\overline{\pi(x)}\subset \pi(x)$, thus $\pi(x)\in C(X)$.\\

  Now we will prove that $\mathcal{R}$ is  closed. Let $(x_n,y_n)_{n\geq 0}$ a sequence of $X^2$ with $x_n\mathcal{R}y_n$, for any $n\geq 0$ we will prove that $x\mathcal{R}y$.\\
  Let $z_n\in X$ such that $\{x_n,y_n\}\subset \pi(z_n)$. If for infinitely many $n\geq 0$ $\pi(z_n)=\pi(z)$ for some $z\in X$ then $\{x,y\}\subset \pi(z)$ thus $x\mathcal{R}y$. So we can assume that for some increasing sequence of integer $(m_n)_{n\geq 0},\; (\pi(z_{m_n}))_{n\geq 0}$ are pairwise disjoint thus $\displaystyle\lim_{n\to +\infty}\delta(\pi((z_{m_n})))=0$, therefore $\displaystyle\lim_{n\to +\infty}d(x_{m_n},y_{m_n})=0$ and so $x=y$, in particular $x\mathcal{R}y$.\\
  Therefore $\mathcal{R}$ is a closed invariant monotone equivalence relation and the claim is proven. Let $(\tilde{X},\tilde{f})$ be the factor system, since $\pi$ is monotone $\tilde{X}$ is also finitely suslinean and $\tilde{f}$ is minimal and weakly mixing, moreover $\tilde{X}$ is not reduced to point since $x\notin \pi(f(x))$ (otherwise $f$ will have a fixed point)

  \textbf{Claim 2} For any $\tilde{C}\in C(\tilde{X})$, such that $\tilde{C}$ is not reduced to a point, $\displaystyle\lim_{n\to +\infty}\delta(\tilde{f}^{n}(\tilde{C}))\neq 0$.\\
  Assume that there exist $\tilde{C}\in C(\tilde{X})$ not reduced to a point, such that $\displaystyle\lim_{n\to +\infty}\delta(\tilde{f}^{n}(\tilde{C}))= 0$ and let $C=\pi^{-1}(\tilde{C})$. Since $\pi$ is monotone $C\in C(X)$, moreover $\displaystyle\lim_{n\to +\infty}\delta(f^{n}(C))\neq 0$, otherwise $C\subset \pi(x)$, for some $x\in C$ and $\pi(C)=\tilde{C}$ is reduced to a point. Therefore the family $(f^{n}(C))_{n\geq 0}$ is not pairwise disjoint and we can find $n<m$ such that $f^{n}(C)\cap f^{m}(C)\neq \emptyset$ and so $\pi(f^{n}(C))\cap \pi(f^{m}(C))\neq \emptyset$. Thus $\tilde{f}^{n}(\tilde{C})\cap \tilde{f}^{m}(\tilde{C})\neq \emptyset$. Thus we can find $\tilde{x_1},\tilde{x_2} \in \tilde{C}$ such that $\tilde{f}^{n}(\tilde{x_1})=\tilde{f}^{m}(\tilde{x_2})$. Thus for any $k\geq 0,\; \tilde{f}^{n+k}(\tilde{x_1})=\tilde{f}^{m+k}(\tilde{x_2})$. Let $(m_k)_{k\geq 0}$ an increasing sequence of integer such that $(\tilde{f}^{m_k}(\tilde{x_1}))_{k\geq 0}$ converges to some $\tilde{c}\in \tilde{X}$. Since $\tilde{C}\times \tilde{C} \subset Asm(\tilde{X},\tilde{f})$, we have that $(\tilde{x_1},\tilde{x_2})\in Asm(\tilde{X},\tilde{f})$, therefore $(f^{m_k}(x_2))_{k\geq 0}$ converges also to $c\in X$. For any $k\geq 0$, we have $\tilde{f}^{n}(\tilde{f}^{m_k}(\tilde{x_1}))=\tilde{f}^{m}(\tilde{f}^{m_k}(\tilde{x_2}))$, thus $\tilde{f}^{n}(\tilde{c})=\tilde{f}^{m}(\tilde{c})$ with $n\neq m$, therefore $P(\tilde{f})\neq \emptyset$ a contradiction and the claim is proven.\\

  Since $\tilde{X}$ is a finitely suslinean continua not reduced to a point and $\tilde{f}$ is minimal and weakly mixing, we can for the rest of the proof and will assume that $(X,f)$ satisfies claim 2, in other word for any $C\in C(X)\setminus S(X),\; \inf_{n\geq 0}\delta (f^n(C))>0$.\\

  Let $A\in C(X)\setminus S(X)$, where $S(X)=\{\{x\},\; x\in X\}$. Denote by $$\epsilon_{A}=\sup_{n\geq 0}\delta (f^n(C)),\; \emph{and} \; \epsilon=\inf_{A\in C(X)\setminus S(X)}\epsilon_{A}.$$
 \textbf{Claim 3} $\epsilon=0$. Otherwise if $\epsilon>0$, then $f$ will be positivity continuum wise expensive with constant $\frac{\epsilon}{2}$. In fact for any $A\in C(X)\setminus S(X), \epsilon_{A}>\frac{\epsilon}{2}$, thus for some $n\geq 0,\; \delta(f^{n}(A))>\frac{\epsilon}{2}$. This end the proof of Claim 3 since $X$ is minimal and connected (see Theorem \ref{PCW}).\\
 Now since $\epsilon=0$ we can find a sequence $(A_n)_{n\geq 0}$ such that $\epsilon_{A_n}<\frac{1}{n}$. For each $n\geq 0$, let $M_n \subset \omega_{C(f)}(A_n)$ be a minimal set for $C(f)$. For any $C\in M_n$ we have that $\delta(C)\leq \frac{1}{n}$, moreover since $\inf_{n\geq 0}\delta (f^n(C))>0$, then for any $D\in M_n$ we have that $0<\delta(D)\leq \frac{1}{n}$. In other word $M_n \cap S(X)=\emptyset$.\\
 Let $k\geq 0$ such that $\frac{1}{k}\leq \frac{\delta(X)}{2}$ and $M=M_k$. Denote by $\mathcal{B}$ a countable base of opens set with countable boundary (Recall that $X$ is in particular a rational continuum) $D=\{a_n,\; n\geq 0\}$ be the union of the boundary of the elements of $\mathcal{B}$. Clearly for any $A\in C(X)\setminus S(X)$, there exists some $n\geq 0$ such that $a_n \in A$. Thus $C(X)\setminus S(X)\subset \displaystyle\bigcup_{n\geq 0}O_n$, where $O_n=\{A\in C(X),\; a_n\in A\}$. Clearly each $O_n$ is a closed subset of $C(X)$, moreover $M=\displaystyle\bigcup_{n\geq 0}(O_n\cap M)$, since $M$ is compact then for some $n\geq 0,\; M\cap O_n$ contains an open set $O\subset O_n$ of $M$. Let $a=a_n$, since $C(f)$ is minimal on $M$, then for some $k\geq 0,\: M= \displaystyle\bigcup_{0\leq i\leq k}C(f)^{i}(O)$. Since $f$ is weakly mixing and minimal we can find $b\in X$ such that $(a,b)$ is a transitive pair of $f^{2}$. In one hand $X= \displaystyle\bigcup_{I\in M}I$, then we can find some $I\in M$ such that $b\in I$, in the other hand $I\in C(f)^{i}(O)$, for some $0\leq i\leq k$, therefore $\{f^{i}(a),b\}\subset I$. Now since $(a,b)$ is a transitive pair of $f^{2}$, then so is $(f^{i}(a),b)$. Let $(x_1,x_2)\in X^2$ such that $d(x_1,x_2)\geq \frac{\delta(X)}{2}$, we can find a sequence of point of the orbit of $(f^{i}(a),b)$ under $f^2$ converging to $(x_1,x_2)$, thus for some $m\geq 0,\; d(f^{m+i}(a),f^{m}(b))\geq \frac{\delta(X)}{2}$, thus $\delta(f^{m}(I))\geq \frac{\delta(X)}{2}$, a contradiction.

\end{proof}

\begin{lem}\label{azz}\cite{az} \rm{
Let $X$ be a continuum, the following assertions are equivalents:\\
 (i) $X$ is hereditarily locally connected;\\
 (ii) For any sequence $(A_n)_{n\geq 0}$ in $C(X)$ converging to $A\in C(X)$, we have $\displaystyle\lim_{n\rightarrow +\infty} Mesh(A_{n}\setminus A)=0$, (where $Mesh(B)=\displaystyle\sup \{diam(C): \ C$ is a connected component of $B\})$.\\
 (iii) For any sequence $(A_n)_{n\geq 0}$ of subcontinua of $X$ converging to a non degenerated subcontinuum $A$ of $X$, we have $\displaystyle\lim_{n\rightarrow +\infty}A_{n}\cap A=A$.}
\end{lem}

\begin{lem}\label{pc}
 Let $\pi: X\to S_1$ such that $X$ is a hereditarily locally connected continuum and $\pi$ is an onto monotone map satisfaying for any $x\in S_1, \pi^{-1}(x)$ has empty interior, then $f$ is an homeomorphism, in particular $X=S_1$.
\end{lem}
 \begin{proof}
   \textbf{Claim} For each $x\in X,\; X\setminus \pi^{-1}(x)$ is connected and we can find a sequence of open sets $(U_n)_{n\geq 0}$ with the following properties:\\
   \begin{enumerate}
    \item $\pi^{-1}(x) \subset U_{n+1} \subset U_n$;
    \item $\text{card}(\partial(U_n)) = 2$;
    \item $\pi^{-1}(x) = \bigcap_{n \in \mathbb{N}} \overline{U_n}$;
    \item $U_n$ is a connected open subset of $X$.
\end{enumerate}

\noindent
\textbf{Proof:} The connectedness of $X \setminus \pi^{1}(x)$ can be easily deduced from the fact that
$\pi$ is monotone from $X$ onto $S_1$ and $S_1\setminus \pi(x)$ is connected. Let $N = \{s \in S_1 \mid \text{diam}(\pi^{-1}(s)) > 0\}$, since $X$ is hereditarily locally connected and $\pi$ is monotone, $N$ is at most countable. It follows that, we may choose a sequence of open arcs $(V_n)_{n \geq 0}$ in $S_1$ with end points $z_n$ and $t_n$ lying outside $\{\tilde{x}\} \cup N$ such that $V_{n+1} \subset V_n \quad \text{and} \quad \{\pi(x)\} = \bigcap_{n \in \mathbb{N}} V_n.$\\
Let, for each $n$, $U_n = \pi^{-1}(V_n)$. Thus, we may verify easily that $(U_n)_{n \geq 0}$ is a sequence of open sets of $X$ with the desired properties. This end the proof of the claim.\\

Let $x\in S_1$, such that $\pi^{-1}(x)$ is not reduced to a point. Recall that $\pi^{-1}(x)$ is a connected subset with empty interior then we can find a sequences $(a^1_n,a^2_n,a^3_n)_{n\geq 0} \in X^3$ such that $(a^1_n,a^2_n,a^3_n)\to (a_1,a_2,a_3)\in \pi^{-1}(x)$ with $\epsilon= \min_{1\leq i<j\leq 3} d(a_i,a_j)>0$. We have $\pi^{-1}(x) = \bigcap_{n \in \mathbb{N}} \overline{U_n}$. By Lemma \ref{azz}, we can find $N\geq 0$ such that $Mesh(U_N\setminus \pi^{-1}{x})<\frac{\epsilon}{2}$, denote by $U=U_N\subsetneq X$.\\
Let $C$ be a connected component of $\overline{U}\setminus \pi^{-1}(x)$ and $y\in C$. Since $X\setminus \pi^{-1}(x)$ is connected we can find an arc $I$ joining $y$ and some point $c\notin \overline{U}$, therefore $C\cap \overline{\partial U} \neq \emptyset$. We have that $\partial U$ contains two point, we conclude that $\overline{U}\setminus \pi^{-1}(x)$ has at most two connected component. Let $n$ large enough so that $\{a^1_n,a^2_n,a^3_n\}\subset U$ and $\min_{1\leq i<j\leq 3} d(a^{n}_i,a^{n}_j)>\frac{\epsilon}{2}$. Since $\overline{U}\setminus \pi^{-1}(x)$ has at most two connected components, let $C$ be the connected component of $\overline{U}\setminus \pi^{-1}(x)$ that contains $a^{n}_i \neq a^{n}_j$ with $1\leq i<j\leq 2$, therefore $diam(C)\geq \frac{\epsilon}{2}$ a contradiction. We conclude that $\pi^{-1}(x)$ is reduced to a point and thus $\pi$ is one to one, therefore an homeomorphism and $X=S_1$.

 \end{proof}

\begin{thm}
  Let $f:X\to X$ be a minimal map on a finitely suslinean continua $X$, then $X=S_1$ and $f$ is an irrational rotation of the circle.
\end{thm}

\begin{proof}
Let $f:X\to X$ be a minimal map on a finitely suslinean continua $X$. By Theorem \ref{NWM}, $f$ is not weakly mixing then by Theorem \ref{RPP}, $f$ has a none-trival maximal equicontinoues factor via a monotone map $\pi$. Since $\pi(X)$ is a finitely suslinean continuum that admit a minimal equicontinous homeomorphism, then it is homogenous and so $\pi(X)=S_1$. It remains to show that for any $x\in X,\; \pi^{-1}(x)$ has an empty interior, this follows immediately since otherwise $\pi^{-1}(x)$ will contains $f^{k}(x)$ for some $k>0$, and so $(x,f^{k}(x))\in RP^+ (X,f)$. Since $f^{n}(RP^{+}(x))_{n\geq 0}$ is a null family we have that $(x,f^{k}(x))$ is an asymptotic pair and therefore $x$ is a periodic point, which is a contradiction. Thus $\pi^{-1}(x)$ has an empty interior and so by lemma \ref{pc}, $X=S_1$ and $f$ is an irrational rotation of the circle.
\end{proof}
\section*{Acknowledgements}
\begin{itemize}
  \item This work was supported by the research unit: “Dynamical systems and their applications”, (UR17ES21),
Ministry of Higher Education and Scientific Research, Faculty of Science of Bizerte, Bizerte, Tunisia.
  \item The author is thankful to Mr.Issam Naghmouchi for his helpful remarks and discussion during the preparation of the paper.
\end{itemize}

\bibliographystyle{amsplain}

\end{document}